\documentclass[11pt]{amsart}
\usepackage{amsmath,amsthm,amssymb,amsfonts,amscd}
\usepackage{amsxtra}
\usepackage{eucal}
\usepackage{mathrsfs}
\usepackage[pdftex]{color, graphicx}
\usepackage[all]{xy}
\usepackage{hyperref}
\usepackage{pgf,tikz}
\usetikzlibrary{arrows}
\usepackage{cases}
\usepackage{arydshln}
\usepackage{fullpage}
\definecolor{forestgreen(traditional)}{rgb}{0.0, 0.27, 0.13}
\definecolor{forestgreen(web)}{rgb}{0.13, 0.55, 0.13}
\definecolor{airforceblue}{rgb}{0.36, 0.54, 0.66}

\hypersetup{
  colorlinks,
  citecolor=forestgreen(web),
  linkcolor=airforceblue,
  urlcolor=magenta}

\headsep=1cm \numberwithin{equation}{section}
\SelectTips{cm}{11}

\newcommand{\defi}[1]{\textsf{#1}} 

\newcommand \im   {\ensuremath{\mathrm{im}}}

\newcommand \Tan {\ensuremath{\mathrm{Tan}}}

\newcommand \Sing{\ensuremath{\mathrm{Sing}}}

\newcommand \codim {\ensuremath{\mathrm{codim}}}

\newcommand \rank {\ensuremath{\mathrm{rank}}}
\newcommand \brank {\ensuremath{\underline{\mathrm{rank}}}}

\newcommand{\SL}{\operatorname{SL}}

\def\P{{\mathbb P}}
\def\PP{{\mathbb P}}
\def\CC{{\mathbb C}}

\def\NN{{\mathbb N}}
\def\RR{{\mathbb R}}

\def\dtwo{\lfloor \frac d2\rfloor}

\newcommand \p[1] {\phi_{#1}}

\newcommand \st[1] {\stackrel{#1}{\longrightarrow}}

\newcommand{\dbchk}{{\bf \textcolor{violet}{** need to be checked later **}}}
\def\ds{\displaystyle}
\newcommand \rrlap[1]{\hbox to 0pt{#1}}

\newcommand \tr[1]{\textcolor{red}{#1}}
\newcommand \tte[1]{\textcolor{teal}{#1}}
\newcommand \tb[1]{\textcolor{brown}{#1}}
\newcommand \tv[1]{\textcolor{violet}{#1}}

\newcommand \ti[1]{\textit{#1}}
\newcommand \trm[1]{\textrm{#1}}

\newcommand \tbf[1]{\textbf{#1}}
\newcommand \mbf[1]{\mathbf{#1}}

\theoremstyle{theorem} 
\newtheorem{Thm}{Theorem}[section]
\newtheorem{Prop}[Thm]{Proposition}
\newtheorem{Coro}[Thm]{Corollary}
\newtheorem{Lem}[Thm]{Lemma}

\theoremstyle{definition}

\newtheorem{Ex}[Thm]{Example}
\newtheorem{Qu}[Thm]{Question}

\newtheorem{Problem}[Thm]{Problem}
\newtheorem{Remk}[Thm]{Remark}

\numberwithin{equation}{section}

\begin{document}

\title{On singularities of third secant varieties of Veronese embeddings}
\author[K.\ Han] {Kangjin Han}
\address{School of Undergraduate Studies,
Daegu-Gyeongbuk Institute of Science \& Technology (DGIST),
333 Techno jungang-daero, Hyeonpung-myeon, Dalseong-gun
Daegu 42988,
Republic of Korea}
\email{kjhan@dgist.ac.kr}

\thanks{The author was supported by Basic Science Research Program through the National Research Foundation
of Korea (grant No. 2012R1A1A2038506), the POSCO Science Fellowship of POSCO TJ Park Foundation, and the DGIST Start-up Fund of the Ministry of Science, ICT and Future Planning (No.\ 2016010066). }



\begin{abstract}
In this paper we study singularities of third secant varieties of Veronese embedding $v_d(\P^n)$, which corresponds to the variety of symmetric tensors of border rank at most three in $(\mathbb{C}^{n+1})^{\otimes d}$. \end{abstract}

\keywords{singularity, secant variety, Veronese embedding, Segre embedding}
\subjclass[2010]{14M12, 14J60, 15A21, 15A69}
\maketitle
\tableofcontents \setcounter{page}{1}

\section{Introduction}\label{intro}

For a projective algebraic variety $X\subset \P W$, the \ti{$k$-th secant variety} $\sigma_k(X)$ is defined
by
\begin{equation}
\sigma_k(X)     = \overline{ \bigcup_{x_1\cdots x_k\in X}\P \langle x_1\cdots x_k\rangle }\subset \P W
\end{equation}
where $\langle x_1\cdots x_k\rangle\subset W$ denotes the linear span of the points $x_1\cdots x_k$ and the overline
denotes Zariski closure. Let $V$ be an $(n+1)$-dimensional complex vector space and $W=S^d V$ be the subspace of symmetric $d$-way tensors in $V^{\otimes d}$. Equivalently, we can also think of $W$ as the space of homogeneous polynomials of degree $d$ in $n+1$ variables. When $X$ is the Veronese embedding $v_d(\P V)$ of rank one symmetric $d$-way tensors over $V$ in $\P W$,  then $\sigma_k(X)$ is the variety of symmetric $d$-way tensors of border rank at most $k$ (see Subsection \ref{prelim} for terminology and details).

If $X$ is an irreducible variety and $\sigma_k(X)$ its $k$-secant variety, then it is well known that 
\begin{equation}\label{sing_ineq}
\Sing(\sigma_k (X))\supseteq \sigma_{k-1}(X)~,
\end{equation}
(e.g. see \cite[Corollary 1.8]{Ad}). Equality holds in many basic examples, like determinantal varieties defined by minors of a generic matrix, but the strict inequality also holds for some other tensors (e.g. just have a look at \cite[Corollary 7.17]{MOZ} for the case $\sigma_2(X)$ when $X$ is the Segre embedding $\P V_1\times\cdots\times\P V_r$ or \cite[Figure 1, p.18]{AOP2} for the third secant variety of Grassmannian $\mathbb{G}(2,6)$).

Therefore, it should be very interesting to compute more cases and to give a general treatment about singularities of secant varieties. Further, the knowledge of singular locus is known to be very crucial to the so-called \ti{identifiablity problem}, which is to determine uniqueness of a tensor decomposition (see \cite[Theorem 4.5]{COV}). It has recently been paid more attention in this context. In this paper, we deal with the case of third secant variety of Veronese embeddings, $\sigma_3(v_d (\P V))$.

From now on, let $X$ be the Veronese variety $v_d(\P V)$ in $\P S^d V=\P^N$ with $N=\dim_\CC S^d V-1={n+d\choose n}-1$. One could ask the following problem:

\begin{Problem}\label{prbm_Vero} Let $V=\CC^{n+1}$. Determine for which triple $(k,d,n)$ it does hold that 
\[\Sing(\sigma_k(v_d(\P V)))=\sigma_{k-1}(v_d(\P V))\]
for every $k\ge2, d\ge2$ and $n\ge1$ or describe $\Sing(\sigma_k(v_d(\P V)))$ if it is not the case.
\end{Problem}

We'd like to remark here that our question is a set-theoretic one. First, it is classical that the answer to Problem \ref{prbm_Vero} is true for the binary case (i.e. $n=1$) (see e.g. \cite[Theorem 1.45]{IK}) and also for the case of quadratic forms (i.e. $d=2$) (see e.g. \cite[Theorem 1.26]{IK}). In the case of $k=2$, Kanev proved in \cite[Theorem 3.3]{Kan} that this holds for any $d,n$. Thus, we only need to take care of the cases of $k\ge3, d\ge3$ and $n\ge2$. Look at the table in Figure \ref{sing3table}.


\section{Singularities of third secant of $v_d(\P^n)$}\label{sect_Vero}

Choose any form $f\in S^{d}V$. We define the \defi{space of essential variables} of $f$ to be $$\langle f\rangle:=\{\partial\in V^{\vee}|\partial(f)=0\}^{\perp}$$ in $V$. So, $f$ also belongs to $S^d \langle f\rangle$ and $\dim\langle f\rangle$ is the minimal number of variables in which we can express $f$ as a homogeneous polynomial of degree $d$ (we also call $\dim\langle f\rangle$ the \defi{number of essential variables} of $f$, see e.g. \cite{Cal}). Note that $\dim\langle f\rangle=1$ means $f\in v_d(\P V)$ by definition. We often abuse $f$ to denote the point $[f]$ in $\P S^d V$ represented by it. We say a form $f\in\sigma_3(X)\setminus	\sigma_2(X)$ to be \ti{degenerate} if $\dim\langle f\rangle=2$ and \ti{non-degenerate} otherwise. We denote the locus of degenerate forms by $\mathcal{D}$. We begin this section by stating our main theorem for the cases of $k=3, d\ge3$ and $n\ge2$.

\begin{Thm}[Singularity of $\sigma_3(v_d(\P^n))$]\label{sing3vero} Let $X$ be the $n$-dimensional Veronese variety $v_d(\P V)$ in $\P^N$ with $N={n+d\choose d}-1$. Then, the following holds that the singular locus
\begin{displaymath}
\Sing(\sigma_3(X))=\sigma_2(X)
\end{displaymath}
as a set for all $(d,n)$ with $d\ge3$ and $n\ge2$ unless $d=4$ and $n\ge3$. In the exceptional case $d=4$, for each $n\ge3$ the singular locus $\Sing(\sigma_3(v_4(\P V)))$ is $\mathcal{D}_4\cup\sigma_2(v_4(\P V))$, where $\mathcal{D}_4$ denotes the locus of all the degenerate forms $f$ (i.e. $\dim\langle f\rangle=2$) in $\sigma_3(v_4(\P V))\setminus\sigma_2(v_4(\P V))$.
\end{Thm}

\begin{figure}[!htb]
$$
\begin{array}{|l|c|c|}
\hline
{\bf (k,d,n)}&{\bf Singular~locus~of~}\sigma_k(v_d(\P^n))&{\bf Comment~\&~Reference}\\
\hline
(\ge2,\ge2,1)&\sigma_{k-1}&\trm{Classical - case of binary forms, \cite[Theorem 1.45]{IK}}\\
\hline
(\ge2,2,\ge1)&\sigma_{k-1}&\trm{Symmetric matrice case, \cite[Theorem 1.26]{IK}}\\
\hline
(2,\ge2,\ge1)&\sigma_{1}&\trm{\cite[Theorem 3.3]{Kan}}\\
\hline
(3,3,\ge 2)&\sigma_{2}&\trm{Aronhold case - Thm. \ref{singAro} ($n=2$), Coro. \ref{thm_d=3} ($n\ge3$)}\\
\hline
(3,\ge4,2)&\sigma_{2}&\trm{Thm. \ref{thm_nondeg}~+~Thm. \ref{thm_deg}}\\
\hline
(3,4,\ge3)&\mathcal{D}_4\cup\sigma_{2}&\trm{Only exceptional case ($d=4$), Thm. \ref{thm_deg}}\\
\hline
(3,\ge5,\ge3)&\sigma_{2}&\trm{Thm. \ref{sing3vero}}\\
\hline
\end{array}
$$
\caption{Singular locus of $\sigma_k(v_d(\P^n))$.}
\label{sing3table}
\end{figure}

\subsection{Preliminaries}\label{prelim}

For the proof, we recall some preliminaries on (border) ranks and geometry of symmetric tensors and list a few known facts on them for future use.

First of all, the scheme-theoretic equations defining $\sigma_3(v_d(\P V))$ come from so-called \defi{symmetric flattenings} unless $d=3$. In the case of $d=3$, we need Aronhold's equation (\ref{sect_Aronhold}) additionally (see e.g. \cite{LO}). Consider the polynomial ring $S^\bullet V=\CC[x_0,\ldots,x_n]$ (we call this ring $S$) and consider another polynomial ring $T=S^\bullet V^\vee=\CC[y_0,\ldots,y_n]$, where $V^\vee$ is the \ti{dual space} of $V$. Define the differential action of $T$ on $S$ as follows: for any $g\in T_{d-k}, f\in S_d$, we set
\begin{equation}
 g\cdot f=g(\partial_0,\partial_1,\ldots,\partial_n)f\in S_k~.
\end{equation}
Let us take bases for $S_k$ and $T_{d-k}$ as
\begin{align}\label{bases}
\mbf{X}^{I}=\frac{1}{i_0 !\cdots i_n !}x_0^{i_0}\cdots x_n^{i_n}&\quad\trm{and}\quad
\mbf{Y}^{J}=y_0^{j_0}\cdots y_n^{j_n}~,
\end{align}
with $|I|=i_0+\cdots+i_n=k$ and $|J|=j_0+\cdots+j_n=d-k$. For a given $f=\sum_{|I|=d}a_I\cdot \mbf{X}^I$ in $S_d$, we have a linear map
\[\phi_{d-k,k}(f):T_{d-k}\to S_k,\quad g\mapsto g\cdot f\] for any $k$ with $1\le k\le d-1$, which can be represented by the following ${k+n\choose n}\times{d-k+n\choose n}$-matrix:
\begin{equation}\label{flatmat}
\left(\begin{array}{ccc}
&&\\
&a_{I,J}&\\
&&\end{array}\right) \quad \trm{with $a_{I,J}=a_{I+J}$}~,
\end{equation}
in the bases defined above. We call this the \ti{symmetric flattening} (or \ti{catalecticant}) of $f$. It is easy to see that the transpose $\phi_{d-k,k}(f)^{T}$ is equal to $\phi_{k,d-k}(f)$.

Given a homogeneous polynomial $f$ of degree $d$, the minimum number of linear forms $l_{i}$ needed to write $f$ as a sum of $d$-th powers is the so-called (Waring) \defi{rank} of $f$ and denoted by $\rank(f)$. The (Waring) \defi{border rank} is this notion in the limiting sense. In other words, if there is a family $\{f_{\epsilon}\mid \epsilon >0 \}$ of polynomials with constant rank $r$ and $\lim_{\epsilon \to 0}f_{\epsilon} = f$, then we say that $f$ has border rank at most $r$. The minimum such $r$ is called the border rank of $f$ and denoted by $\brank(f)$. Note that by definition $\sigma_k(v_d(\P V))$ is the variety of homogeneous polynomials $f$ of degree $d$ with border rank $\brank(f)\le k$.

It is obvious that if $f$ has (border) rank 1, then any symmetric flattening $\phi_{d-k,k}(f)$ has rank 1. By subadditivity of matrix rank, we also know that $\rank~\phi_{d-k,k}(f)\le r$ if $\brank(f)\le r$. So, we could obtain a set of defining equations coming from minors of the matrix $\phi_{d-k,k}(f)$ for $\sigma_r(v_d(\P V))$. For $\sigma_3(v_d(\P V))$ and $d\ge4$, it is known that these minors are sufficient to cut it scheme-theoretically in \cite[Theorem 3.2.1 (1)]{LO} as follows:

\begin{Prop}[Defining equations of $\sigma_3(v_d(\P^n))$]\label{eqn_s3} Let $X$ be the $n$-dimensional Veronese variety $v_d(\P V)$ in $\P^N$ with $N={n+d\choose n}-1$. For any $(d,n)$ with $d\ge4, n\ge2$, $\sigma_3(X)$ is defined scheme-theoretically by the $4\times 4$-minors of the two symmetric flattenings
\[\phi_{d-1,1}(F)\colon {S^{d-1}V}^{\vee}\to V\quad\trm{and}\quad\phi_{d-\dtwo,\dtwo}(F)\colon {S^{d-\dtwo}V}^{\vee}\to S^{\dtwo}V~,\]
where $F$ is the form $\ds\sum_{I\in\NN^{n+1}} a_I\cdot\mbf{X}^I$ of degree $d$ as considering the coefficients $a_I$'s indeterminate.
\end{Prop}

Since there is a natural $\SL_{n+1}(\CC)$-group action on $\sigma_3(X)$, we may take the $\SL_{n+1}(\CC)$-orbits inside $\sigma_3(X)$ into consideration for the study of singularity. And we could also regard a canonical representative of each orbit as below. 

First, suppose $f\in\sigma_3(X)\setminus\sigma_2(X)$ is a degenerate form (i.e. $\dim\langle f\rangle=2$). Choose $x_0, x_1$ as a basis of $\langle f\rangle$. Then, we recall the following lemma

\begin{Lem}\label{deg_normal} For any $d\ge4$ and $n\ge1$, any general degenerate form $f\in \sigma_3(v_d(\P V))\setminus\sigma_2(v_d(\P V))$ can be written as $x_0^{d}+\alpha \cdot x_1^{d}+\beta\cdot (x_0+x_1)^{d}$, up to $\SL_{n+1}(\CC)$-action, for some nonzero $\alpha, \beta\in\CC$.
\end{Lem}
\begin{proof} Since $\dim\langle f\rangle=2$, let $U:=\langle f\rangle=\CC\langle x_0,x_1\rangle$, a subspace of $V$. For such a $f\in \sigma_3(v_d(\P V))\setminus\sigma_2(v_d(\P V))$, it is easy to see that \[3=\brank(f)\le\brank(f,U)~,\] where the latter is the border rank of $f$ being considered as a polynomial in $S^\bullet U$. On the other hand, we also have $\brank(f,U)\le\brank(f)=3$, because $U\subset V$ implies that $\sigma_3(v_d(\mathbb{P} U))$ is contained in $\sigma_3(v_d(\mathbb{P} V))$. Since $\rank(f,U)$ and $\brank(f,U)$ coincide for a \ti{general} $f$ in the rational normal curve case (see e.g. \cite{CG}), we have $\rank(f,U)=3$. Thus, for some nonzero $\lambda,\mu\in\CC$ we can write $f$ as
\begin{align*}
f(x_0,x_1)&=(a_0 x_0+a_1 x_1)^{d}+(b_0 x_0+b_1 x_1)^d +\{\lambda(a_0 x_0+a_1 x_1)+\mu(b_0 x_0+b_1 x_1)\}^{d}\\
&=X_0^{d}+(\frac{\lambda}{\mu})^d\cdot X_1^d+\lambda^d\cdot(X_0+X_1)^d~,
\end{align*}
by some scaling and using a $\SL_{n+1}(\CC)$-change of coordinates, which proves our assertion.\end{proof}

\begin{Remk}\label{d<=3} There are some remarks related to Lemma \ref{deg_normal} as follows:
\begin{itemize}
\item[(a)] Note that there does not exist a degenerate form corresponding to an orbit in $\sigma_3(v_d(\P V))\setminus\sigma_2(v_d(\P V))$ if $d\le 3$. In this case, if $f$ is degenerate, then $f$ always belongs to $\sigma_2(v_d(\P V))$, for the $\phi_{d-1,1}(f)$ have at most two nonzero rows and all the $3\times3$-minors of $\phi_{d-1,1}(f)$ vanish.
\item[(b)] In fact, in $d=4$ case, Lemma \ref{deg_normal} holds for \ti{all} degenerate form $f\in\sigma_3(v_4(\P V))\setminus\sigma_2(v_4(\P V))$, because there exist only $\rank~3$ forms in $\sigma_3(v_4(\P^1))\setminus\sigma_2(v_4(\P^1))$ (see \cite{CG} and also \cite[Section 4]{LT}).
\end{itemize}
\end{Remk}

Now, let's put main types of canonical representatives for $\SL_{n+1}(\CC)$-orbits together as follows:

\begin{Thm}\label{normal_form} There are 4 types of homogeneous forms representing $\SL_{n+1}(\CC)$-orbits in $\sigma_3(v_d(\P V))\setminus\sigma_2(v_d(\P V))$; 
\begin{enumerate}
\item $x_0^{d}+x_1^{d}+x_2^{d}$
\item $x_0^{d-1}x_1+x_2^{d}$
\item $x_0^{d-2}x_1^{2}+x_0^{d-1}x_2$
\item $x_0^{d}+\alpha x_1^{d}+\beta(x_0+x_1)^{d}$ (for some nonzero $\alpha, \beta\in\CC$)~.
\end{enumerate}
The first three types correspond to all the three non-degenerate orbits. And the last `binary' type corresponds to a general point of $\mathcal{D}$, the locus of all the degenerate forms, which appears only if $d\ge4$.
\end{Thm}
\begin{proof}
It is straightforward from \cite[Theorem 10.2]{LT}, Lemma \ref{deg_normal} and Remark \ref{d<=3}.
\end{proof}

Let us introduce more basic terms and facts. Let $Z\subset\P W$ be a variety and $\hat{Z}$ be its affine cone in $W$. Consider a (closed) point $p\in \hat{Z}$ and say $[p]$ the corresponding point in $\P W$. We denote the \ti{affine tangent space to $Z$ at $[p]$} in $W$ by $\hat{T}_{[p]}Z$ and we define the  \defi{(affine) conormal space to $Z$ at $[p]$}, $\hat{N}^{\vee}_{[p]}Z$ as the annihilator $(\hat{T}_{[p]}Z)^{\perp}\subset W^{\vee}$. Since $\dim \hat{N}^{\vee}_{[p]}Z+\dim\hat{T}_{[p]}Z=\dim W$ and $\dim Z\le \dim \hat{T}_{[p]}Z-1$, we get that $\dim \hat{N}^{\vee}_{[p]}Z\le \codim(Z,\P W)$ and the equality holds if and only if $Z$ is smooth at $[p]$. This conormal space is quite useful to study the tangent space of $Z$.

Let us recall the \defi{apolar ideal} $f^{\perp}\subset T$. For any given form $f\in S^d V$, we call $\partial\in T_t$ \ti{apolar} to $f$ if the differentiation $\partial(f)$ gives zero (i.e. $\partial\in\ker\phi_{t,d-t}(f)$). And we define the \defi{apolar ideal} $f^{\perp}\subset T$ as
\[f^\perp:=\{\partial\in T~|~\partial(f)=0\}~.\]
It is straightforward to see that $f^\perp$ is indeed an ideal of $T$. Moreover, it is well-known that the quotient ring $T_f:=T/f^\perp$ is an \ti{Artinian Gorenstein algebra with socle degree $d$} (see e.g. \cite{IK}).

In our case, we have a nice description of the conormal space in terms of this apolar ideal as follows:

\begin{Prop}\label{conormal_prop} Let $X$ be the $n$-dimensional Veronese variety $v_d(\P V)$ as above and $f$ be any form in $S^d V$. Suppose that $f$ corresponds to a (closed) point of $\sigma_3(X)\setminus\sigma_2(X)$ and that $\rank~\phi_{d-1,1}(f)=3,~\rank~\phi_{d-\dtwo,\dtwo}(f)=3$. Then, for any $(d,n)$ with $d\ge4, n\ge2$ we have
\begin{equation}\label{conormal1}
\hat{N}^{\vee}_{f}\sigma_3(X)=(f^\perp)_1\cdot(f^\perp)_{d-1}+(f^\perp)_{\dtwo}\cdot(f^\perp)_{d-\dtwo}~,
\end{equation}
where the sum is taken as a $\CC$-subspace in $T_d={S^d V}^{\vee}$.
\end{Prop}
\begin{proof}
First, recall that $\phi_{d-k,k}(f)^{T}=\phi_{k,d-k}(f)$. We also note that 
\[\ker \phi_{d-k,k}(f)=(f)^\perp_{d-k}\quad\trm{and}\quad(\im~\phi_{d-k,k}(f))^\perp=\ker (\phi_{d-k,k}(f)^{T})=\ker \phi_{k,d-k}(f)=(f)^\perp_{k}~.\]
Whenever $\rank~\phi_{d-1,1}(f)=3$ and $\rank~\phi_{d-\dtwo,\dtwo}(f)=3$, we have
\begin{equation}\label{conormal}
\hat{N}^{\vee}_{f}\sigma_3(X)=\langle\ker \phi_{d-1,1}(f)\cdot(\im~\phi_{d-1,1}(f))^{\perp}\rangle+\langle\ker \phi_{d-\dtwo,\dtwo}(f)\cdot(\im~\phi_{d-\dtwo,\dtwo}(f))^{\perp}\rangle
\end{equation}
(see \cite[Proposition 2.5.1]{LO}), which proves the proposition.
\end{proof}

\begin{Remk}\label{deg_conormal} Note that, in case of $n=2$ or $\dim\langle f\rangle=2$ (i.e. degenerate form), to compute conormal space $\hat{N}^{\vee}_{f}\sigma_3(X)$ we only need to consider the symmetric flattening $\phi_{d-\dtwo,\dtwo}$ so that we have 
\begin{equation}\label{conormal2}
\hat{N}^{\vee}_{f}\sigma_3(X)=(f^\perp)_{\dtwo}\cdot(f^\perp)_{d-\dtwo}~.
\end{equation}
For $n=2$ case, $\phi_{d-1,1}(f)$ has only 3 rows, there is no non-trivial $4\times4$-minor to give a local equation of $\sigma_3(X)$ at $f$. In case of $\dim\langle f\rangle=2$, we may consider $f\in\CC[x_0,x_1]_d$ and choose bases as (\ref{bases}). Then, we could write the matrix of $\phi_{d-1,1}$ and its evaluation at $f$, $\phi_{d-1,1}(f)$ as
\begin{align*}
\phi_{d-1,1}=\left(\begin{array}{c:cccccc}
&y_0^{d-1}&y_0^{d-2}y_1&\cdots&y_n^{d-1}\\ \hdashline
x_0&&&&\\
x_1&&&&\\
x_2&&a_I&&\\
\vdots&&&\\
x_n&&&
\end{array}\right)~,&\quad\phi_{d-1,1}(f)=
\left(\begin{array}{c:cccccc}
&y_0^{d-1}&y_0^{d-2}y_1&\cdots&y_n^{d-1}\\ \hdashline
x_0&\ast&\ast&\cdots&\ast\\
x_1&\ast&\ast&\cdots&\ast\\
x_2&0&0&\cdots&0\\
\vdots&\vdots&\vdots&\vdots&\vdots\\
x_n&0&0&\cdots&0
\end{array}\right)~.
\end{align*}
So, each $4\times4$-minor of $\phi_{d-1,1}$ (say $D_4(\phi_{d-1,1})$) has at most rank $2$ at $f$. Hence, we see that all the partial derivatives in the Jacobian
\[\frac{\partial D_4(\phi_{d-1,1})}{\partial a_I}(f)=0\]
for each index $I$ with $|I|=d$ and $D_4(\phi_{d-1,1})$ doesn't contribute to span the conormal space of $\sigma_3(X)$ at $f$, because at least one row of $D_4(\phi_{d-1,1})$ (say $(a_I~a_J~a_K~a_L)$) vanishes at $f$ and the Laplace expansion of $D_4(\phi_{d-1,1})$ along this row
\[D_4(\phi_{d-1,1})=\pm\bigg(a_I\cdot D^I_3(\phi_{d-1,1})-a_J\cdot D^J_3(\phi_{d-1,1})+a_K\cdot D^K_3(\phi_{d-1,1})-a_L\cdot D^L_3(\phi_{d-1,1})\bigg)\]
guarantees all the partials of $D_4(\phi_{d-1,1})$ become zero at $f$ as follows: for example, we see that
\begin{align*}
\pm\frac{\partial D_4(\phi_{d-1,1})}{\partial a_I}(f)=&~ D^I_3(\phi_{d-1,1})(f)+a_I(f)\cdot\frac{\partial D^I_3(\phi_{d-1,1})}{\partial a_I}(f)-a_J(f)\cdot\frac{\partial D^J_3(\phi_{d-1,1})}{\partial a_I}(f)\\
&+a_K(f)\cdot\frac{\partial D^K_3(\phi_{d-1,1})}{\partial a_I}(f)-a_L(f)\cdot\frac{\partial D^L_3(\phi_{d-1,1})}{\partial a_I}(f)~=~0~,
\end{align*}
\end{Remk}
where $a_I(f)=a_J(f)=a_K(f)=a_L(f)=0$ and $D^I_3(\phi_{d-1,1})(f)=0$ because of $\rank~D^I_3(\phi_{d-1,1})$ is at most 2 at $f$.

\subsection{Outline for the proof of main theorem}

In this subsection we outline the proof of our main theorem (Theorem~\ref{sing3vero}).

For the locus of non-degenerate orbits in $\sigma_3(X)\setminus\sigma_2(X)$, we may consider a useful reduction method through the following arguments:

\begin{Lem}\label{span_lem} For every $f\in \sigma_3(v_d(\P^n))$ ($d,n\ge 2$), there exists a linear
$\P^2=\P U\subset\P^n=\P V$ such that $f\in \sigma_3(v_d(\P U))$. In particular, for every $f\in \sigma_3(v_d(\P^n))\setminus\sigma_2(v_d(\P^n))$, $2\le\dim\langle f\rangle\le 3$. 
\end{Lem}
\begin{proof}
It is enough to show that $\dim\langle f\rangle\le 3$. When $f\in\sigma_3(v_d(\P^n))$ (i.e. border rank $\le 3$), the image of the flattening
$\phi_{d-1,1} : {S^{d-1}\CC^{n+1}}^{\vee}\to\CC^{n+1}$ has dimension $\le 3$ and it is contained in the required 3-dimensional subspace $U$, i.e. $\dim\langle f\rangle\le 3$.
\end{proof}

Recall that we denote the locus of degenerate forms in $\sigma_3(X)\setminus\sigma_2(X)$ by $\mathcal{D}$ (see the paragraph before Theorem \ref{sing3vero} for notation). Then, by Lemma \ref{span_lem}, we have an obvious corollary for non-degnerate orbits as follows:

\begin{Coro}\label{coro_span} For each $f\in\sigma_3(v_d(\P^n))\setminus\left(\mathcal{D}\cup\sigma_2(v_d(\P^n))\right)$, There exists a unique 3-dimensional subspace $U$ such that $f\in \sigma_3(v_d(\P U))$.
\end{Coro}
\begin{proof} For any non-degenerate form $f$, which correspond to three orbits in Theorem \ref{normal_form}, the dimension of $\langle f\rangle$ is exactly $3$ and the subspace $U=\langle f\rangle$ is uniquely determined as the image of the flattening $\phi_{d-1,1}$.
\end{proof}

For the proof of the main theorem, we treat the case of non-degenerate forms and the case of degenerate ones separately:

\begin{proof}[Proof of Theorem~\ref{sing3vero}]
Let our $\P^n=\P V$ with $V=\CC\langle x_0,x_1,\cdots,x_n\rangle$ and its dual $V^{\vee}=\CC\langle y_0,y_1,\cdots,y_n\rangle$. First, for the locus of non-degenerate forms, we claim that one may reduce the problem to the case of $n=2$. Construct the following map
$$\sigma_3(v_d(\PP^n))\setminus\left(\mathcal{D}\cup\sigma_2(v_d(\PP^n))\right)\st{\pi}\mathbb{G}(2,\PP^n)~.$$
This map is well defined by Corollary \ref{coro_span} and for each 2-dimensional $\PP U\subset\PP^n$, the fiber $\pi^{-1}(\PP U)$ is isomorphic to
$\sigma_3(v_d(\PP U))\setminus\left(\mathcal{D}\cup\sigma_2(v_d(\PP U))\right)$. So, if we prove our theorem for the case $n=2$, then the fibers of $\pi$ are all isomorphic and smooth. Hence $\pi$ becomes a fibration over a smooth variety with smooth fibers. This shows that its domain
$\sigma_3(v_d(\PP^n))\setminus\left(\mathcal{D}\cup\sigma_2(v_d(\PP^n))\right)$ is smooth, so proving our assertion.

Thus, in subsection \ref{sect_n=2} we investigate the non-degenerate orbits with condition $n=2$ and prove that there are no more singularity than $\sigma_2(v_d(\PP^n))$ by Corollary \ref{thm_d=3} ($d=3$) and Theorem \ref{thm_nondeg} ($d\ge4$). 

For the locus of degenerate forms $\mathcal{D}$, in subsection \ref{sect_deg} we directly compute the dimension of conormal space  $\hat{N}^{\vee}\sigma_3(v_d(\PP^n))$ using Proposition \ref{conormal_prop} and show that $\mathcal{D}$ happens to be the extra singular locus only when $d=4,~n\ge3$ (see Theorem \ref{thm_deg}).
\end{proof}	
\bigskip

\subsection{Non-degenerate orbits : $n=2$ case}\label{sect_n=2}

\subsubsection{Aronhold case ($d=3$)}\label{sect_Aronhold} Here, we settle the equality in $\Sing(\sigma_3(v_d(\P V)))\supseteq\sigma_{2}(v_d(\P V))$ in our first case $d=3,~\dim V=3$ (i.e. $n=2$). Note that the equation for the hypersurface $\sigma_3(v_3(\P^2))\subset \P^9$ is given by 
the classical \textit{Aronhold invariant} (see e.g. \cite{Ott, LO}). Map $S^3V\rightarrow (V\otimes \Lambda^2 V)\otimes (V\otimes V^*)$, by
first embedding $S^3V\subset V\otimes V\otimes V$, then tensoring with $Id_V\in V\otimes V^*$, and then skew-symmetrizing.
Thus, $F\in S^3V$ gives rise to an element of $\CC^9\otimes \CC^9$. In suitable bases, if we write
\begin{align*}
F= &\phi_{000}x_0^3+\phi_{111}x_1^3+\phi_{222}x_2^3+3\phi_{001}x_0^2x_1+3\phi_{011}x_0x_1^2+3\phi_{002}x_0^2x_2\\
&+3\phi_{022}x_0x_2^2
+3\phi_{112}x_1^2x_2+3\phi_{122}x_1x_2^2+6\phi_{012}x_0x_1x_2,
\end{align*} 
then the corresponding skew-symmetric matrix is:
$$\left[\begin{array}{ccccccccc}
&&&\phi_{002}&\phi_{012}&\phi_{022}&-\phi_{001}&-\phi_{011}&-\phi_{012}\\
&&&\phi_{012}&\phi_{112}&\phi_{122}&-\phi_{011}&-\phi_{111}&-\phi_{112}\\
&&&\phi_{022}&\phi_{122}&\phi_{222}&-\phi_{012}&-\phi_{112}&-\phi_{122}\\
-\phi_{002}&-\phi_{012}&-\phi_{022}&&&&\phi_{000}&\phi_{001}&\phi_{002}\\
-\phi_{012}&-\phi_{112}&-\phi_{122}&&&&\phi_{001}&\phi_{011}&\phi_{012}\\
-\phi_{022}&-\phi_{122}&-\phi_{222}&&&&\phi_{002}&\phi_{012}&\phi_{022}\\
\phi_{001}&\phi_{011}&\phi_{012}&-\phi_{000}&-\phi_{001}&-\phi_{002}&&&\\
\phi_{011}&\phi_{111}&\phi_{112}&-\phi_{001}&-\phi_{011}&-\phi_{012}&&&\\
\phi_{012}&\phi_{112}&\phi_{122}&-\phi_{002}&-\phi_{012}&-\phi_{022}&&&\\
\end{array}\right].
$$
All the principal Pfaffians of size $8$ of the this matrix coincide with one another (up to scaling), this quartic equation gives the classical Aronhold invariant $A(F)$ as follows:
\begin{align*}
A(F)&=\p{012}^4-2\p{011}\p{012}^2\p{022}+\p{011}^2\p{022}^2+\p{002}\p{012}\p{022}\p{111}-\p{001}\p{022}^2\p{111}-2\p{002}\p{012}^2\p{112}\\
&-\p{002}\p{011}\p{022}\p{112}+3\p{001}\p{012}\p{022}\p{112}+\p{002}^2\p{112}^2-\p{000}\p{022}\p{112}^2+3\p{002}\p{011}\p{012}\p{122}\\
&-2\p{001}\p{012}^2\p{122}-\p{001}\p{011}\p{022}\p{122}-\p{002}^2\p{111}\p{122}+\p{000}\p{022}\p{111}\p{122}-\p{001}\p{002}\p{112}\p{122}\\
&+\p{000}\p{012}\p{112}\p{122}+\p{001}^2\p{122}^2-\p{000}\p{011}\p{122}^2-\p{002}\p{011}^2\p{222}+\p{001}\p{011}\p{012}\p{222}\\
&+\p{001}\p{002}\p{111}\p{222}-\p{000}\p{012}\p{111}\p{222}-\p{001}^2\p{112}\p{222}+\p{000}\p{011}\p{112}\p{222}~.
\end{align*}
\begin{Thm}\label{singAro}
The singular locus $\Sing(\sigma_3(v_3(\P^2)))$ coincides with $\sigma_{2}(v_3(\P^2))$ set-theoretically.
\end{Thm}
\begin{proof}
We know $\Sing(\sigma_3(v_3(\P^2)))\supseteq\sigma_{2}(v_3(\P^2))$. It is also well-known that the defining equations of $\sigma_{2}(v_3(\P^2))$ are given by $3$-minors of 3 by 6 catalecticant matrix $\phi_{2,1}$ (e.g. \cite{Kan}), which are ${6\choose 3}=20$ cubics cutting out degree 15 and codimension 4 variety.

On the other hand, the Jacobian of $A(F)$ gives 10 cubic equations, which cut out the singular locus $\Sing(\sigma_3(v_3(\P^2)))$, such as

\begin{align*}
g_1=& \p{011}\p{112}\p{222}-\p{011}\p{122}^2-\p{012}\p{111}\p{222}+\p{012}\p{112}\p{122}+\p{022}\p{111}\p{122}-\p{022}\p{112}^2\\ 
g_2=&-2\p{001}\p{112}\p{222}+2\p{001}\p{122}^2+\p{002}\p{111}\p{222}-\p{002}\p{112}\p{122}+\p{011}\p{012}\p{222}-\p{011}\p{022}\p{122}\\
&-2\p{012}^2\p{122}+3\p{012}\p{022}\p{112}-\p{022}^2\p{111}\\
   g_3=& \p{001}\p{111}\p{222}-\p{001}\p{112}\p{122}-2\p{002}\p{111}\p{122}+2\p{002}\p{112}^2-\p{011}^2\p{222}+3\p{011}\p{012}\p{122}\\
   &-\p{011}\p{022}\p{112}-2\p{012}^2\p{112}+\p{012}\p{022}\p{111}\\  
\end{align*}   
\begin{align*}   
  g_4=& \p{000}\p{112}\p{222}-\p{000}\p{122}^2+\p{001}\p{012}\p{222}-\p{001}\p{022}\p{122}-2\p{002}\p{011}\p{222}+3\p{002}\p{012}\p{122}\\
   &-\p{002}\p{022}\p{112}+2\p{011}\p{022}^2-2\p{012}^2\p{022}\\
   g_5 = & \p{000}\p{111}\p{222}-\p{000}\p{112}\p{122}-\p{001}\p{011}\p{222}+4\p{001}\p{012}\p{122}-3\p{001}\p{022}\p{112}-3\p{002}\p{011}\p{122}\\
   &4\p{002}\p{012}\p{112}-\p{002}\p{022}\p{111}+4\p{011}\p{012}\p{022}-4\p{012}^3\\
    g_6 =& \p{000}\p{111}\p{122}-\p{000}\p{112}^2-\p{001}\p{011}\p{122}+3\p{001}\p{012}\p{112}-2\p{001}\p{022}\p{111}-\p{002}\p{011}\p{112}\\
    &+\p{002}\p{012}\p{111}+2\p{011}^2\p{022}-2\p{011}\p{012}^2\\
     g_7=&  \p{000}\p{012}\p{222}-\p{000}\p{022}\p{122}-\p{001}\p{002}\p{222}+\p{001}\p{022}^2+\p{002}^2\p{122}-\p{002}\p{012}\p{022}\\
     g_8=&   \p{000}\p{011}\p{222}+\p{000}\p{012}\p{122}-2\p{000}\p{022}\p{112}-\p{001}^2\p{222}-\p{001}\p{002}\p{122}+3\p{001}\p{012}\p{022}\\
     &+2\p{002}^2\p{112}-\p{002}\p{011}\p{022}-2\p{002}\p{012}^2\\
  g_9  =& -2\p{000}\p{011}\p{122}+\p{000}\p{012}\p{112}+\p{000}\p{022}\p{111}+2\p{001}^2\p{122}-\p{001}\p{002}\p{112}-\p{001}\p{011}\p{022}\\
  &-2\p{001}\p{012}^2-\p{002}^2\p{111}+3\p{002}\p{011}\p{012}\\
   g_{10}= & \p{000}\p{011}\p{112}-\p{000}\p{012}\p{111}-\p{001}^2\p{112}+\p{001}\p{002}\p{111}+\p{001}\p{011}\p{012}-\p{002}\p{011}^2 \quad.
\end{align*}

One can compute the Hilbert polynomial of the singular locus by these ten cubics (e.g. \cite{M2}) as follows:
$$\mathrm{H}(t)=\frac{15}{5!}t^5+\frac{15}{8}t^4-\frac{53}{24}t^3+\frac{81}{8}t^2-\frac{23}{12}t+2~,$$
which shows that $\Sing(\sigma_3(v_3(\P^2)))$ has also codimension 4 in $\P^9$ and degree 15. This gives the equality $\Sing(\sigma_3(v_3(\P^2)))=\sigma_{2}(v_3(\P^2))$ as a set.
\end{proof}

Thus, we also have an immediate corollary as follows:
\begin{Coro}[$d=3$ case]\label{thm_d=3} For every $n\ge 2$ and $d=3$, 
$\sigma_3(v_3(\P^n))\setminus\sigma_2(v_3(\P^n))$ is smooth.
\end{Coro}
\begin{proof}
By Remark \ref{d<=3} (a), there is no degenerate orbit in this case. Thus, it comes directly from the result on the Aronhold hypersurface (i.e. $n=2$ case, Theorem \ref{singAro}) and using the fibration argument in the proof of Theorem~\ref{sing3vero} for any $n\ge3$.
\end{proof}

\subsubsection{Cases of non-degenerate orbits ($d\ge4$)}

Here is the theorem for non-degenerate orbits for any $d\ge4$ and $n=2$:
\begin{Thm}[Non-degenerate locus]\label{thm_nondeg} For every $d\ge4$ and $n=2$, 
$\sigma_3(v_d(\P^n))\setminus\left(\mathcal{D}\cup\sigma_2(v_d(\P^n))\right)$ is smooth.
\end{Thm}

\begin{proof}
We are enough to consider three different cases according to Theorem \ref{normal_form}. It is well-known that $\dim \sigma_3(v_d(\P^2))$ is $3\cdot2+2=8$, the expected one, for any $d\ge4$ (see e.g. \cite{AH}).\\

Case (i) $f_1=x_0^{d}+x_1^{d}+x_2^{d}$ (Fermat-type). It is well-known that this Fermat-type $f_1$ becomes an almost transitive $\SL_{3}(\CC)$-orbit, which corresponds to a general point of $\sigma_3(v_d(\PP^2))$, Thus, $\sigma_3(v_d(\PP^2))$ is smooth at $f_1$.

Case (ii) $f_2=x_0^{d-1}x_1+x_2^d$ (Unmixed-type). Say $X=v_d(\PP^2)$. By Remark \ref{deg_conormal} (i.e. $n=2$ case), we just need to consider $(f_2^\perp)_{\dtwo}\cdot(f_2^\perp)_{d-\dtwo}$ as (\ref{conormal2}) to compute $\dim \hat{N}^{\vee}_{f_2}\sigma_3(X)$. Say $s:=\dtwo$. For $d\ge4$, we have $2\le s \le d-s\le d-2$. Note that $\dim \hat{N}^{\vee}_{f_2}\sigma_3(X)\le \codim(\sigma_3(X),\P U)={d+2\choose 2}-9$. So, it is enough to show $\dim \hat{N}^{\vee}_{f_2}\sigma_3(X)\ge{d+2\choose 2}-9$ for proving non-singularity of $f_2$.

Since the summands of $f_2$ separate the variables (i.e. unmixed-type), we could see that the apolar ideal $f_2^\perp$ is generated as
\[f_2^\perp=\bigg(\{Q_1=y_0 y_2, Q_2=y_1^2, Q_3=y_1 y_2\}\bigcup~\{\trm{other generators in degree $\ge d$}\}\bigg)~.\]
So, we have
\begin{align*}
&(f_2^\perp)_s=\{h\cdot Q_i~|~\forall h\in T_{s-2},~i=1,2,3~\}~\trm{and}~(f_2^\perp)_{d-s}=\{h^\prime\cdot Q_i~|~\forall h^\prime\in T_{d-s-2},~i=1,2,3~\}\\
&\Rightarrow\quad\hat{N}^{\vee}_{f_2}\sigma_3(X)=(f_2^\perp)_s \cdot (f_2^\perp)_{d-s}=\{h^{\prime\prime}\cdot Q_i Q_j~|~\forall h^{\prime\prime}\in T_{d-4},~i,j=1,2,3~\}~.
\end{align*}
Thus, if we denote the ideal $(Q_1,Q_2,Q_3)$ by $I$, then $\dim\hat{N}^{\vee}_{f_2}\sigma_3(X)$ is equal to the value of Hilbert function $H(I^2,t)$ at $t=d$ (concentrating only on degree $d$-part of $(f_2^\perp)_s \cdot (f_2^\perp)_{d-s}$, other generators in degree $\ge d$ do not affect the dimension computation). But, it is easy to see that $I^2$ has a minimal free resolution as
\[0\to T(-6)\to T(-5)^6\to T(-4)^6\to I^2\to 0~,\]
which shows the Hilbert function of $I^2$ can be computed as
\begin{align*}
H(I^2,d)&=6{d-4+2\choose 2}-6{d-5+2\choose 2}+{d-6+2\choose 2}\\
&=\left\{ \begin{array}{ll}
 0 & \textrm{$(d\le3)$}\\
&\\
 {d+2\choose 2}-9 & \textrm{$(d\ge4)$}
  \end{array} \right.~.
\end{align*}
This implies that $\dim\hat{N}^{\vee}_{f_2}\sigma_3(X)={d+2\choose 2}-9$ for any $d\ge4$, which means that our $\sigma_3(X)$ is smooth at $f_2$ (see also Figure \ref{Mink1}).
\begin{figure}[!hbt]
\begin{align*}
\definecolor{zzzzff}{rgb}{0.6,0.6,1}
\definecolor{ffqqtt}{rgb}{1,0,0.2}
\definecolor{ttttff}{rgb}{0.2,0.2,1}
\begin{tikzpicture}[line cap=round,line join=round,>=triangle 45,x=2.0cm,y=2.0cm]
\clip(-0.3,-0.44) rectangle (2.32,2.22);
\fill[line width=1.6pt,color=zzzzff,fill=zzzzff,fill opacity=0.25] (0.98,1.73) -- (0.1,0.18) -- (0.23,0) -- (1.78,0) -- (1.78,0.37) -- cycle;
\draw [line width=0.4pt] (0.98,1.73)-- (0,0);
\draw [line width=0.4pt] (0,0)-- (2,0);
\draw [line width=0.4pt] (2,0)-- (0.98,1.73);
\draw [line width=1.6pt,color=zzzzff] (0.98,1.73)-- (0.1,0.18);
\draw [line width=1.6pt,color=zzzzff] (0.1,0.18)-- (0.23,0);
\draw [line width=1.6pt,color=zzzzff] (0.23,0)-- (1.78,0);
\draw [line width=1.6pt,color=zzzzff] (1.78,0)-- (1.78,0.37);
\draw [line width=1.6pt,color=zzzzff] (1.78,0.37)-- (0.98,1.73);
\draw (0.75,0.73) node[anchor=north west] {\large$P_1$};
\draw [->] (0.98,1.73) -- (0.98,2.1);
\draw [->] (0,0) -- (-0.29,-0.21);
\draw [->] (2,0) -- (2.3,-0.23);
\draw (0.7,2.23) node[anchor=north west] {$ j $};
\draw (-0.3,0.15) node[anchor=north west] {$ i $};
\draw (2.16,0.15) node[anchor=north west] {$ k $};
\draw (-0.03,0.02) node[anchor=north west] {$d-s$};
\begin{scriptsize}
\fill [color=ttttff] (0.98,1.73) circle (2.5pt);
\draw [color=ffqqtt] (0,0) circle (2.5pt);
\draw[color=ffqqtt] (2,0) circle (2.5pt);
\draw[color=ffqqtt] (1.89,0.18) circle (2.5pt);
\fill [color=ttttff] (1.78,0) circle (2.5pt);
\fill [color=ttttff] (1.78,0.37) circle (2.5pt);
\fill [color=ttttff] (0.23,0) circle (2.5pt);
\fill [color=ttttff] (0.1,0.18) circle (2.5pt);
\fill [color=ttttff] (0.88,1.55) circle (2.5pt);
\fill [color=ttttff] (1.08,1.55) circle (2.5pt);
\fill [color=ttttff] (0.22,0.38) circle (2.5pt);
\fill [color=ttttff] (0.36,0.18) circle (2.5pt);
\fill [color=ttttff] (0.47,0) circle (2.5pt);
\fill [color=ttttff] (1.67,0.17) circle (2.5pt);
\fill [color=ttttff] (1.53,0) circle (2.5pt);
\end{scriptsize}
\end{tikzpicture}
\definecolor{ffwwtt}{rgb}{1,0.4,0.2}
\definecolor{ffqqtt}{rgb}{1,0,0.2}
\definecolor{ccttqq}{rgb}{0.8,0.2,0}
\begin{tikzpicture}[line cap=round,line join=round,>=triangle 45,x=2.0cm,y=2.0cm]
\clip(-0.5,-0.44) rectangle (2.32,2.22);
\fill[line width=1.6pt,color=ffwwtt,fill=ffwwtt,fill opacity=0.25] (0.97,1.3) -- (0.28,0.16) -- (0.38,0) -- (1.66,0) -- (1.65,0.31) -- cycle;
\draw [line width=0.4pt] (0.97,1.3)-- (0.19,0);
\draw [line width=0.4pt] (0.19,0)-- (1.87,0);
\draw [line width=0.4pt] (1.87,0)-- (0.97,1.3);
\draw [line width=1.6pt,color=ffwwtt] (0.97,1.3)-- (0.28,0.16);
\draw [line width=1.6pt,color=ffwwtt] (0.28,0.16)-- (0.38,0);
\draw [line width=1.6pt,color=ffwwtt] (0.38,0)-- (1.66,0);
\draw [line width=1.6pt,color=ffwwtt] (1.66,0)-- (1.65,0.31);
\draw [line width=1.6pt,color=ffwwtt] (1.65,0.31)-- (0.97,1.3);
\draw (0.75,0.69) node[anchor=north west] {\large$P_2$};
\draw [->] (0.97,1.3) -- (0.98,1.8);
\draw [->] (0.19,0) -- (-0.29,-0.21);
\draw [->] (1.87,0) -- (2.3,-0.23);
\draw (0.7,1.91) node[anchor=north west] {$ j $};
\draw (-0.25,0.11) node[anchor=north west] {$ i $};
\draw (2.13,0.13) node[anchor=north west] {$k $};
\draw (0.13,0.01) node[anchor=north west] {$s$};
\begin{scriptsize}
\fill [color=ttttff] (0.97,1.3) circle (2.5pt);
\draw [color=ffqqtt] (1.87,0) circle (2.5pt);
\draw [color=ffqqtt] (0.19,0) circle (2.5pt);
\draw [color=ffqqtt] (1.76,0.15) circle (2.5pt);
\fill [color=ttttff] (1.66,0) circle (2.5pt);
\fill [color=ttttff] (1.65,0.31) circle (2.5pt);
\fill [color=ttttff] (0.38,0) circle (2.5pt);
\fill [color=ttttff] (0.28,0.16) circle (2.5pt);
\fill [color=ttttff] (0.89,1.17) circle (2.5pt);
\fill [color=ttttff] (1.06,1.17) circle (2.5pt);
\fill [color=ttttff] (0.39,0.33) circle (2.5pt);
\fill [color=ttttff] (0.49,0.17) circle (2.5pt);
\fill [color=ttttff] (0.58,0) circle (2.5pt);
\fill [color=ttttff] (1.55,0.16) circle (2.5pt);
\fill [color=ttttff] (1.45,0) circle (2.5pt);
\end{scriptsize}
\end{tikzpicture}
\definecolor{qqqqff}{rgb}{0,0,1}
\definecolor{qqzzzz}{rgb}{0,0.6,0.6}
\definecolor{ffqqtt}{rgb}{1,0,0.2}
\definecolor{ttttff}{rgb}{0.2,0.2,1}
\begin{tikzpicture}[line cap=round,line join=round,>=triangle 45,x=2.0cm,y=2.0cm]
\clip(-0.6,-0.44) rectangle (2.72,2.30);
\fill[line width=1.6pt,color=qqzzzz,fill=qqzzzz,fill opacity=0.25] (0.98,1.82) -- (0.11,0.34) -- (0.4,-0.03) -- (1.61,-0.04) -- (1.61,0.75) -- cycle;
\draw [line width=0.4pt] (0.98,1.82)-- (-0.11,-0.03);
\draw [line width=0.4pt] (-0.11,-0.03)-- (2.08,-0.04);
\draw [line width=0.4pt] (2.08,-0.04)-- (0.98,1.82);
\draw [line width=1.6pt,color=qqzzzz] (0.98,1.82)-- (0.11,0.34);
\draw [line width=1.6pt,color=qqzzzz] (0.11,0.34)-- (0.4,-0.03);
\draw [line width=1.6pt,color=qqzzzz] (0.4,-0.03)-- (1.61,-0.04);
\draw [line width=1.6pt,color=qqzzzz] (1.61,-0.04)-- (1.61,0.75);
\draw [line width=1.6pt,color=qqzzzz] (1.61,0.75)-- (0.98,1.82);
\draw (0.56,0.89) node[anchor=north west] {\large$P_1+P_2$};
\draw [->] (0.98,1.82) -- (0.99,2.31);
\draw [->] (-0.11,-0.03) -- (-0.41,-0.3);
\draw [->] (2.08,-0.04) -- (2.35,-0.32);
\draw (0.75,2.29) node[anchor=north west] {$j$};
\draw (-0.41,0.11) node[anchor=north west] {$i$};
\draw (2.19,0.11) node[anchor=north west] {$k$};
\draw (-0.19,-0.02) node[anchor=north west] {$d$};
\begin{scriptsize}
\fill [color=ttttff] (0.98,1.82) circle (2.5pt);
\draw [color=ffqqtt] (-0.11,-0.03) circle (2.5pt);
\draw [color=ffqqtt] (2.08,-0.04) circle (2.5pt);
\draw[color=ffqqtt] (1.86,0.32) circle (2.5pt);
\fill [color=ttttff] (1.61,-0.04) circle (2.5pt);
\fill [color=ttttff] (1.61,0.75) circle (2.5pt);
\fill [color=ttttff] (0.4,-0.03) circle (2.5pt);
\fill [color=ttttff] (0.11,0.34) circle (2.5pt);
\fill [color=ttttff] (0.87,1.63) circle (2.5pt);
\fill [color=ttttff] (1.09,1.63) circle (2.5pt);
\fill [color=ttttff] (0.25,0.58) circle (2.5pt);
\fill [color=ttttff] (0.4,0.35) circle (2.5pt);
\fill [color=ttttff] (0.68,-0.03) circle (2.5pt);
\fill [color=ttttff] (1.59,0.31) circle (2.5pt);
\draw [color=ffqqtt] (1.75,0.52) circle (2.5pt);
\draw[color=ffqqtt] (1.97,0.13) circle (2.5pt);
\draw[color=ffqqtt] (1.87,-0.04) circle (2.5pt);
\draw[color=ffqqtt] (1.73,0.14) circle (2.5pt);
\fill [color=qqqqff] (0.25,0.15) circle (2.5pt);
\fill [color=qqqqff] (0.54,0.16) circle (2.5pt);
\draw[color=ffqqtt] (0,0.15) circle (2.5pt);
\draw[color=ffqqtt] (0.16,-0.03) circle (2.5pt);
\end{scriptsize}
\end{tikzpicture}
\end{align*}
\caption{Case of $f_2=x_0^{d-1}x_1+x_2^d$. $P_1$ is the lattice polytope in $\RR_{\ge0}^3$ consisting of exponent vectors $(i,j,k)$ of the monomials $y_0^i y_1^k y_2^k$ in $(f_2^\perp)_{d-s}$ and $P_2$ is the one corresponding to $(f_2^\perp)_{s}$. $P_1+P_2$ is the Minkowski sum of two polytopes whose lattice points are exactly the exponent vectors of $\hat{N}^{\vee}_{f_2}\sigma_3(X)=(f_2^\perp)_{d-s}\cdot(f_2^\perp)_{s}$, which 
contains all the monomial of $T_d$ but 9 monomials $y_0^d,~~ y_0^{d-1}y_1,~~ y_0^{d-2}y_1^2,~~ y_0^{d-3}y_1^3,~~ y_0^{d-1}y_2,~~ y_0 y_2^{d-1},~~ y_2^d,~~ y_1 y_2^{d-1},~~ y_0^{d-2}y_1 y_2$. This also shows $\dim\hat{N}^{\vee}_{f_2}\sigma_3(X)={d+2\choose d}-9$.}
\label{Mink1}
\end{figure}
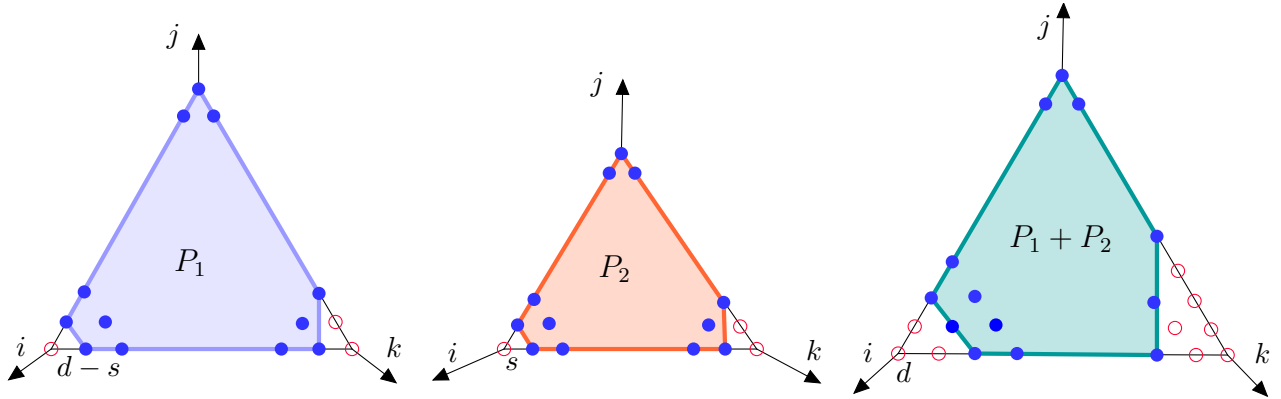

Case (iii) $f_3=x_0^{d-2}x_1^{2}+x_0^{d-1}x_2$ (Mixed-type). In this case, we similarly use a computation of $\dim \hat{N}^{\vee}_{f_3}\sigma_3(X)$ via $(f_3^\perp)_{s}\cdot(f_3^\perp)_{d-s}$ to show the smoothness of $f_3$ (recall $s:=\dtwo$ and $2\le s \le d-s\le d-2$).

Let $Q_1:=y_0y_2-\frac{d-1}{2}y_1^2\in T_2$. We easily see that 
\[f_3^\perp=\bigg(\{Q_1, Q_2=y_1 y_2, Q_3=y_2^2\}\bigcup~\{\trm{other generators in degree $\ge d-1$}\}\bigg)~.\]
Let $I$ be the ideal generated by three quadrics $Q_1,Q_2,Q_3$. By the same reasoning as (ii), we have
\begin{align*}
\dim\hat{N}^{\vee}_{f_3}\sigma_3(X)&=\dim(f_3^\perp)_s \cdot (f_3^\perp)_{d-s}= H(I^2,d)=\left\{ \begin{array}{ll}
 0 & \textrm{$(d\le3)$}\\
&\\
 {d+2\choose 2}-9 & \textrm{$(d\ge4)$}
  \end{array} \right.~,
\end{align*}
because in this case $I^2$ also has the same minimal free resolution $0\to T(-6)\to T(-5)^6\to T(-4)^6\to I^2\to 0$. Hence, we obtain the smoothness of $\sigma_3(X)$ at $f_3$ (see also Figure \ref{Mink2}).

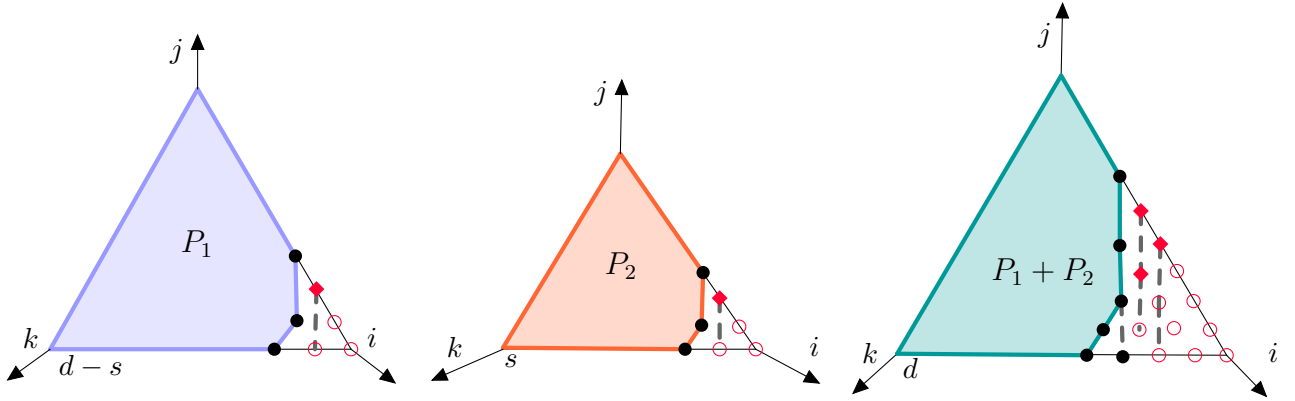
\begin{figure}[!htb]
\begin{align*}
\definecolor{wwwwww}{rgb}{0.4,0.4,0.4}
\definecolor{zzzzff}{rgb}{0.6,0.6,1}
\definecolor{ffqqtt}{rgb}{1,0,0.2}
\begin{tikzpicture}[line cap=round,line join=round,>=triangle 45,x=2.0cm,y=2.0cm]
\clip(-0.3,-0.44) rectangle (2.32,2.22);
\fill[line width=1.6pt,color=zzzzff,fill=zzzzff,fill opacity=0.25] (0.98,1.73) -- (0,0) -- (1.49,0) -- (1.64,0.19) -- (1.63,0.62) -- cycle;
\draw [line width=0.4pt] (0.98,1.73)-- (0,0);
\draw [line width=0.4pt] (0,0)-- (2,0);
\draw [line width=0.4pt] (2,0)-- (0.98,1.73);
\draw [line width=1.6pt,color=zzzzff] (0.98,1.73)-- (0,0);
\draw [line width=1.6pt,color=zzzzff] (0,0)-- (1.49,0);
\draw [line width=1.6pt,color=zzzzff] (1.49,0)-- (1.64,0.19);
\draw [line width=1.6pt,color=zzzzff] (1.64,0.19)-- (1.63,0.62);
\draw [line width=1.6pt,color=zzzzff] (1.63,0.62)-- (0.98,1.73);
\draw (0.8,0.86) node[anchor=north west] {\large$P_1$};
\draw [->] (0.98,1.73) -- (0.98,2.1);
\draw [->] (0,0) -- (-0.29,-0.21);
\draw [->] (2,0) -- (2.3,-0.23);
\draw (0.73,2.13) node[anchor=north west] {$ j $};
\draw (-0.25,0.21) node[anchor=north west] {$ k $};
\draw (2.04,0.21) node[anchor=north west] {$ i $};
\draw (-0.01,0.03) node[anchor=north west] {$d-s$};
\draw [line width=1.6pt,dash pattern=on 5pt off 5pt,color=wwwwww] (1.77,0.4)-- (1.76,0);
\begin{scriptsize}
\draw [color=ffqqtt] (2,0) circle (2.5pt);
\draw [color=ffqqtt] (1.89,0.18) circle (2.5pt);
\fill [color=black] (1.63,0.62) circle (2.5pt);
\fill [color=black] (1.49,0) circle (2.5pt);
\fill [color=ffqqtt] (1.77,0.4) ++(-3pt,0 pt) -- ++(3pt,3pt)--++(3pt,-3pt)--++(-3pt,-3pt)--++(-3pt,3pt);
\draw [color=ffqqtt] (1.76,0) circle (2.5pt);
\fill [color=black] (1.64,0.19) circle (2.5pt);
\end{scriptsize}
\end{tikzpicture}
\definecolor{wwwwww}{rgb}{0.4,0.4,0.4}
\definecolor{ffwwtt}{rgb}{1,0.4,0.2}
\definecolor{ffqqtt}{rgb}{1,0,0.2}
\definecolor{ttttff}{rgb}{0.2,0.2,1}
\begin{tikzpicture}[line cap=round,line join=round,>=triangle 45,x=2.0cm,y=2.0cm]
\clip(-0.5,-0.44) rectangle (2.32,2.22);
\fill[line width=1.6pt,color=ffwwtt,fill=ffwwtt,fill opacity=0.25] (0.97,1.3) -- (0.19,0.01) -- (1.4,0) -- (1.51,0.16) -- (1.52,0.51) -- cycle;
\draw [line width=0.4pt] (0.97,1.3)-- (0.19,0);
\draw [line width=0.4pt] (0.19,0)-- (1.87,0);
\draw [line width=0.4pt] (1.87,0)-- (0.97,1.3);
\draw [line width=1.6pt,color=ffwwtt] (0.97,1.3)-- (0.19,0.01);
\draw [line width=1.6pt,color=ffwwtt] (0.19,0.01)-- (1.4,0);
\draw [line width=1.6pt,color=ffwwtt] (1.4,0)-- (1.51,0.16);
\draw [line width=1.6pt,color=ffwwtt] (1.51,0.16)-- (1.52,0.51);
\draw [line width=1.6pt,color=ffwwtt] (1.52,0.51)-- (0.97,1.3);
\draw (0.8,0.71) node[anchor=north west] {\large$P_2$};
\draw [->] (0.97,1.3) -- (0.98,1.8);
\draw [->] (0.19,0) -- (-0.29,-0.21);
\draw [->] (1.87,0) -- (2.3,-0.23);
\draw (0.73,1.85) node[anchor=north west] {$j$};
\draw (-0.25,0.15) node[anchor=north west] {$k$};
\draw (2.16,0.15) node[anchor=north west] {$i$};
\draw (0.13,0.03) node[anchor=north west] {$s$};
\draw [line width=1.6pt,dash pattern=on 5pt off 5pt,color=wwwwww] (1.63,0.34)-- (1.63,0);
\begin{scriptsize}
\draw[color=ffqqtt] (1.87,0) circle (2.5pt);
\draw [color=ffqqtt] (1.76,0.15) circle (2.5pt);
\fill [color=black] (1.52,0.51) circle (2.5pt);
\fill [color=black] (1.4,0) circle (2.5pt);
\fill [color=ffqqtt] (1.63,0.34) ++(-3pt,0 pt) -- ++(3pt,3pt)--++(3pt,-3pt)--++(-3pt,-3pt)--++(-3pt,3pt);
\draw [color=ffqqtt] (1.63,0) circle (2.5pt);
\fill [color=black] (1.51,0.16) circle (2.5pt);
\end{scriptsize}
\end{tikzpicture}
\definecolor{zzqqff}{rgb}{0.6,0,1}
\definecolor{qqzztt}{rgb}{0,0.6,0.2}
\definecolor{ffwwqq}{rgb}{1,0.4,0}
\definecolor{ttttff}{rgb}{0.2,0.2,1}
\definecolor{qqzzzz}{rgb}{0,0.6,0.6}
\definecolor{ffqqtt}{rgb}{1,0,0.2}
\begin{tikzpicture}[line cap=round,line join=round,>=triangle 45,x=2.0cm,y=2.0cm]
\clip(-0.6,-0.44) rectangle (2.72,2.32);
\fill[line width=1.6pt,color=qqzzzz,fill=qqzzzz,fill opacity=0.25] (0.98,1.82) -- (1.37,1.15) -- (1.37,0.69) -- (1.38,0.32) -- (1.26,0.13) -- (1.15,-0.04) -- (-0.11,-0.03) -- cycle;
\draw [line width=0.4pt] (0.98,1.82)-- (-0.11,-0.03);
\draw [line width=0.4pt] (-0.11,-0.03)-- (2.08,-0.04);
\draw [line width=0.4pt] (2.08,-0.04)-- (0.98,1.82);
\draw (0.45,0.67) node[anchor=north west] {\large$P_1+P_2$};
\draw [->] (0.98,1.82) -- (0.99,2.31);
\draw [->] (-0.11,-0.03) -- (-0.41,-0.3);
\draw [->] (2.08,-0.04) -- (2.35,-0.32);
\draw (0.76,2.24) node[anchor=north west] {$j$};
\draw (-0.41,0.11) node[anchor=north west] {$k$};
\draw (2.29,0.11) node[anchor=north west] {$i$};
\draw (-0.14,0.01) node[anchor=north west] {$d$};
\draw [line width=1.6pt,color=qqzzzz] (0.98,1.82)-- (1.37,1.15);
\draw [line width=1.6pt,color=qqzzzz] (1.37,1.15)-- (1.37,0.69);
\draw [line width=1.6pt,color=qqzzzz] (1.37,0.69)-- (1.38,0.32);
\draw [line width=1.6pt,color=qqzzzz] (1.38,0.32)-- (1.26,0.13);
\draw [line width=1.6pt,color=qqzzzz] (1.26,0.13)-- (1.15,-0.04);
\draw [line width=1.6pt,color=qqzzzz] (1.15,-0.04)-- (-0.11,-0.03);
\draw [line width=1.6pt,color=qqzzzz] (-0.11,-0.03)-- (0.98,1.82);
\draw [line width=1.6pt,dash pattern=on 5pt off 5pt,color=wwwwww] (1.64,0.7)-- (1.63,0.31);
\draw [line width=1.6pt,dash pattern=on 5pt off 5pt,color=wwwwww] (1.63,0.31)-- (1.63,-0.04);
\draw [line width=1.6pt,dash pattern=on 5pt off 5pt,color=wwwwww] (1.51,0.92)-- (1.51,0.5);
\draw [line width=1.6pt,dash pattern=on 5pt off 5pt,color=wwwwww] (1.51,0.5)-- (1.5,0.13);
\draw [line width=1.6pt,dash pattern=on 5pt off 5pt,color=wwwwww] (1.38,0.32)-- (1.39,-0.05);
\begin{scriptsize}
\draw[color=ffqqtt] (2.08,-0.04) circle (2.5pt);
\draw[color=ffqqtt] (1.86,0.32) circle (2.5pt);
\draw[color=ffqqtt] (1.63,-0.04) circle (2.5pt);
\fill[color=ffqqtt] (1.64,0.7) ++(-3pt,0 pt) -- ++(3pt,3pt)--++(3pt,-3pt)--++(-3pt,-3pt)--++(-3pt,3pt);
\fill[color=ffqqtt] (1.51,0.92) ++(-3pt,0 pt) -- ++(3pt,3pt)--++(3pt,-3pt)--++(-3pt,-3pt)--++(-3pt,3pt);
\fill[color=ffqqtt] (1.51,0.5) ++(-3pt,0 pt) -- ++(3pt,3pt)--++(3pt,-3pt)--++(-3pt,-3pt)--++(-3pt,3pt);
\draw[color=ffqqtt] (1.75,0.52) circle (2.5pt);
\draw[color=ffqqtt] (1.97,0.13) circle (2.5pt);
\draw[color=ffqqtt] (1.87,-0.04) circle (2.5pt);
\draw[color=ffqqtt] (1.73,0.14) circle (2.5pt);
\fill [color=black] (1.37,0.69) circle (2.5pt);
\fill [color=black] (1.37,1.15) circle (2.5pt);
\fill [color=black] (1.38,0.32) circle (2.5pt);
\draw [color=ffqqtt] (1.5,0.13) circle (2.5pt);
\draw [color=ffqqtt] (1.63,0.31) circle (2.5pt);
\fill [color=black] (1.39,-0.05) circle (2.5pt);
\fill [color=black] (1.26,0.13) circle (2.5pt);
\fill [color=black] (1.15,-0.04) circle (2.5pt);
\end{scriptsize}
\end{tikzpicture}
\end{align*}
\caption{Case of $f_3=x_0^{d-2}x_1^{2}+x_0^{d-1}x_2$. $P_1$ (resp. $P_2$) is the lattice polytope consisting of exponent vectors $(i,j,k)$ of the monomials $y_0^i y_1^k y_2^k$ in $(f_3^\perp)_{d-s}$ (resp. in $(f_3^\perp)_{s}$). A \ti{dashed} line means an equivalent relation between monomials given by the multiples of $Q_1$ in $P_1$ and $P_2$ and by those of $Q_i Q_j$'s in $P_1+P_2$. The quotient space $T_d/(f_3^\perp)_{d-s}\cdot(f_3^\perp)_{s}$ can be represented by 9 \ti{circle} monomials in $P_1+P_2$ $y_0^d,~~ y_0^{d-1}y_1,~~ y_0^{d-2}y_1^2,~~ y_0^{d-3}y_1^3,~~ y_0^{d-1}y_2,~~ y_0^{d-2} y_2^{2},~~ y_0^{d-3} y_1^2y_2,~~ y_0^{d-3}y_1 y_2^{2},~~ y_0^{d-2}y_1 y_2$ modulo \ti{dashed} relations, which says $\dim\hat{N}^{\vee}_{f_3}\sigma_3(X)={d+2\choose d}-9$, the non-singularity at $f_3$.}
\label{Mink2}
\end{figure}
\end{proof}

\begin{Remk} From the viewpoint of apolarity, the three cases in Theorem \ref{thm_nondeg} can be explained geometrically as follows: if we consider the base locus of the ideal $I$, which is generated by the three quadrics in each apolar ideal $f_i^\perp$, then case (i) corresponds to three distinct points, case (ii) to one reduced point and one non-reduced of length 2, and case (iii) to one non-reduced point of length 3 (not lying on a line).\end{Remk}

\subsection{Degenerate case : binary forms}\label{sect_deg}

Since there is no degenerate form for $d=3$ (see Remark \ref{d<=3} (a)), it is enough to consider the smoothness of the degenerate locus for $d\ge4$.

\begin{Thm}[Degenerate locus]\label{thm_deg} Let $\mathcal{D}$ be the locus of all the degenerate forms in $\sigma_3(v_d(\P^n))\setminus\sigma_2(v_d(\P^n))$. Then, for any $d\ge4, n\ge2$, $\sigma_3(v_d(\P^n))$ is singular on $\mathcal{D}$ if and only if $d=4$ and $n\ge3$.
\end{Thm}

\begin{proof} Let $f_D$ be any form belong to $\mathcal{D}$. For this degenerate case, by Remark \ref{deg_conormal}, we have
\[\hat{N}^{\vee}_{f_D}\sigma_3(X)=(f_D^\perp)_{\dtwo}\cdot(f_D^\perp)_{d-\dtwo}~.\]
First of all, let us consider $f_D$ as a polynomial in $\CC[x_0,x_1]$ (i.e. $f_D=f_D(x_0,x_1)$). Then, by the Hilbert-Burch Theorem (see e.g. \cite[Theorem 1.54]{IK}) we know that $T/f_D^\perp$ is an Artinian Gorenstein algebra with socle degree $d$ and that $f_D^\perp$ is a complete intersection of two homogeneous polynomials $F, G$ of each degree $a$ and $b$ with $a+b=d+2$ \ti{as an ideal of $\CC[y_0,y_1]$}. Since $\rank~\phi_{d-3,3}(f_D)=3$, there is 1-dimensional kernel of $\phi_{d-3,3}(f_D)$ in $\CC[y_0,y_1]_3$, which gives one \ti{cubic} generator $F$ in $f_D^\perp$.

When $f_D$ is general, $f_D=x_0^{d}+\alpha x_1^{d}+\beta (x_0+x_1)^{d}$ for some $\alpha,\beta\in\CC^\ast$ by Lemma \ref{deg_normal}, so we have $F=y_0^2 y_1-y_0 y_1^2$. Even for the case $f_D$ being not general, we have $F=y_0^2 y_1$ up to change of coordinates, because the apolar ideal of this non-general $f_D$ corresponds to the case with one multiple root on $\P^1$ (see \cite{CG} and also \cite[Section 4]{LT}).

Therefore, we obtain that 
\[\trm{$f_D^\perp=\big(\trm{$F=y_0^2 y_1-y_0 y_1^2$ or $y_0^2 y_1$},G\big)$ for some polynomial $G$ of degree $(d-1)$}\]
and that $f_D^\perp$ \ti{as an ideal in $T=\CC[y_0,y_1,\ldots,y_n]$} has its degree parts $(f_D^\perp)_{\dtwo}$ and $(f_D^\perp)_{d-\dtwo}$, both of which are generated by $F,y_2,\ldots,y_n$, since $d\ge4$ so that $\dtwo, d-\dtwo<d-1$.

Now, let us compute the dimension of conormal space as follows:\\

i) $d=4$ case (i.e. $\dtwo=2$) : In this case, we have
\[\hat{N}^{\vee}_{f_D}\sigma_3(X)=(f_D^\perp)_{2}\cdot(f_D^\perp)_{2}=(y_2,\ldots,y_n)_2\cdot(y_2,\ldots,y_n)_2=(\{y_i y_j~|~2\le i,j\le n\})_4~.\] So, we get
\begin{align*}
\dim \hat{N}^{\vee}_{f_D}\sigma_3(X)&=\dim T_4 -\dim\big\langle y_0^4,y_0^3 y_1,\cdots,y_1^4\big\rangle-\dim\big\langle\{y_0^3\cdot\ell, y_0^2 y_1\cdot\ell, y_0 y_1^2\cdot\ell, y_1^3\cdot\ell~|~\ell=y_2,\ldots,y_n\}\big\rangle\\
&={4+n\choose 4}-5-4(n-1)~.
\end{align*}
This shows us that $\sigma_3(X)$ is singular at $f_D$ if and only if $n\ge3$, because the expected codimension is ${4+n\choose 4}-3n-3$.

ii) $d=5$ case (i.e. $\dtwo=2$) : Recall that $F$ is $y_0^2 y_1-y_0 y_1^2$ or $y_0^2 y_1$, the cubic generator of $f_D^\perp$. Then,
\[\hat{N}^{\vee}_{f_D}\sigma_3(X)=(f_D^\perp)_{2}\cdot(f_D^\perp)_{3}=(y_2,\ldots,y_n)_2\cdot(F,y_2,\ldots,y_n)_3~.\]
\begin{align*}
\dim \hat{N}^{\vee}_{f_D}\sigma_3(X)&=\dim T_5 -\dim\big\langle y_0^5,y_0^4 y_1,\cdots,y_1^5\big\rangle\\
&\quad-\dim\bigg\langle\{y_0^4\cdot\ell, y_0^3 y_1\cdot\ell, y_0^2 y_1^2\cdot\ell, y_0 y_1^3\cdot\ell, y_1^4\cdot\ell\}\setminus\{y_0 F\cdot\ell, y_1 F\cdot\ell~|~\ell=y_2,\ldots,y_n\}\bigg\rangle\\
&={5+n\choose 5}-6-3(n-1)=\trm{expected $\codim(\sigma_3(X),\P S^5 V)$}~,
\end{align*}
which gives that $\sigma_3(X)$ is smooth at $f_D$ in this case.

ii) $d\ge6$ case : Here we have $\hat{N}^{\vee}_{f_D}\sigma_3(X)=(f_D^\perp)_{\dtwo}\cdot(f_D^\perp)_{d-\dtwo}=(F,y_2,\ldots,y_n)_{\dtwo}\cdot(F,y_2,\ldots,y_n)_{d-\dtwo}~.$
\begin{align*}
\dim \hat{N}^{\vee}_{f_D}\sigma_3(X)&=\dim T_d -\dim\bigg\langle\{y_0^{d-1}\cdot\ell, y_0^{d-2} y_1\cdot\ell, \ldots, y_1^{d-1}\cdot\ell\}\setminus\{y_0^{d-4}F\cdot\ell,\ldots,y_1^{d-4}F\cdot\ell~|~\ell=y_2,\ldots,y_n\}\bigg\rangle\\
&\quad-\dim\bigg(\{y_0^d,y_0^{d-1} y_1,\cdots,y_1^d\}\setminus\{y_0^{d-6}\cdot F^2,y_0^{d-7}y_1\cdot F^2,\ldots,y_1^{d-6}\cdot F^2\}\bigg)\\
&={d+n\choose d}-\big\{d-(d-3)\big\}(n-1)-\big\{(d+1)-(d-5)\big\}\\&={d+n\choose d}-3(n-1)-6=\trm{expected $\codim(\sigma_3(X),\P S^d V)$}~,
\end{align*}
which implies that $\sigma_3(X)$ is also smooth at $f_D$.
\end{proof}

\subsection{Defining equations of $\Sing(\sigma_3(X))$} As an immediate corollary of Theorem \ref{sing3vero}, we obtain defining equations of the singular locus in our third secant of Veronese embedding $\sigma_3(X)$.

\begin{Coro} Let $X$ be the $n$-dimensional Veronese embedding as above. The singular locus of $\sigma_3(X)$ is cut out by $3\times 3$-minors of the two symmetric flattenings $\phi_{d-1,1}$ and $\phi_{d-2,2}$ unless $d=4$ and $n\ge3$ case, in which the (set-theoretic) defining ideal of the locus is the intersection of the ideal generated by the previous $3\times 3$-minors and the ideal generated by $3\times 3$-minors of $\phi_{d-1,1}$ and $4\times 4$-minors of $\phi_{d-\dtwo,\dtwo}$.
\end{Coro}
\begin{proof}
It is well-known that $\sigma_2(X)$ is cut out by $3\times 3$-minors of the two $\phi_{d-1,1}$ and $\phi_{d-2,2}$ (see \cite[Theorem 3.3]{Kan}). It is also easy to see that $\mathcal{D}$, the locus of degenerate forms inside $\sigma_3(X)$, is cut out by $3\times 3$-minors of $\phi_{d-1,1}$ and $4\times 4$-minors of $\phi_{d-\dtwo,\dtwo}$ by the argument in Remark \ref{deg_conormal} and Proposition \ref{eqn_s3}. Thus, using these two facts the conclusion is straightforward by Theorem \ref{sing3vero}.
\end{proof}

\paragraph*{\textbf{Acknowledgements}} The author would like to express his deep gratitude to Giorgio Ottaviani for introducing the problem, giving many helpful suggestions to him, and encouraging him to complete this work. He also gives thanks to Luca Chiantini and Luke Oeding for useful conversations with them and anonymous referees for their appropriate and accurate comments.

\end{document}